\documentclass[a4paper,oneside,reqno]{amsart}
\usepackage[utf8]{inputenc}

\usepackage{amssymb,amsmath,amsthm}
\usepackage{mathtools}
\usepackage{bbm}
\usepackage{hyperref}
\usepackage[capitalize]{cleveref}
\usepackage{stmaryrd}

\usepackage{geometry}
\newtheorem{theorem}{Theorem}[section]
\newtheorem{lemma}[theorem]{Lemma}

\newcommand{\br}[1]{\llbracket{#1}\rrbracket}

\renewcommand{\le}{\leqslant}
\renewcommand{\ge}{\geqslant}

\renewcommand{\Pr}{\mathbb{P}}
\renewcommand{\emptyset}{\varnothing}

\newcommand{\cG}{\mathcal{G}}
\newcommand{\cH}{\mathcal{H}}
\newcommand{\cL}{\mathcal{L}}

\newcommand{\bR}{\mathbf{R}}
\newcommand{\eps}{\varepsilon}
\newcommand{\1}{\mathbbm{1}}

\DeclareMathOperator{\Ex}{\mathbb{E}}

\DeclareMathOperator{\Bin}{Bin}

\author{Ilay Hoshen}
\address{School of Mathematical Sciences, Tel Aviv University, Tel Aviv 6997801, Israel}
\email{ilayhoshen@gmail.com}

\author{Wojciech Samotij}
\address{School of Mathematical Sciences, Tel Aviv University, Tel Aviv 6997801, Israel}
\email{samotij@tauex.tau.ac.il}

\thanks{This research was supported by the Israel Science Foundation grant 2110/22 and the ERC Consolidator Grant 101044123 (RandomHypGra).}

\title[{Optimal lower bounds for epsilon-nets}]{Optimal lower bounds for epsilon-nets \\ for lines in the plane}
\date{\today}

\begin{document}

\begin{abstract}
  We prove that, for arbitrarily small positive $\varepsilon$, there is a finite planar point set~$P$ such that every $\varepsilon$-net for the
range space induced on $P$ by straight lines has cardinality $\Omega\bigl(1/\varepsilon \cdot \log(1/\varepsilon)\bigr)$.
  This matches the classical upper bound for range spaces with bounded VC-dimension due to Haussler and Welzl and confirms a prediction of Alon.
\end{abstract}

\maketitle

\section{Introduction}
\label{sec:introduction}

Suppose that $\mathcal{H}$ is a hypergraph (often called a \emph{range space} in our context) on a finite set~$V$ and let $\varepsilon$ be a positive real.
An \emph{$\varepsilon$-net} for $\mathcal{H}$ is a set $S \subseteq V$ that intersects every $A \in \mathcal{H}$ with $|A|\ge \varepsilon|V|$.
A fundamental question is how small an $\varepsilon$-net one can always find under various assumptions on the complexity of $\mathcal{H}$.
One well-known measure of complexity is the \emph{Vapnik--Chervonenkis dimension}, introduced in~\cite{CheVap1971}.
We say that a set $S\subseteq V$ is \emph{shattered} by $\mathcal{H}$ if the family $\{A \cap S : A \in \mathcal{H}\}$ comprises all $2^{|S|}$ subsets of $S$.
The \emph{VC-dimension} of $\mathcal{H}$ is the largest cardinality of a shattered set.
The celebrated theorem of Haussler and Welzl~\cite{HauWel1987} states that every hypergraph of VC-dimension at most $d$ admits an $\eps$-net of cardinality $O\bigl(d/\varepsilon \cdot \log(1/\varepsilon)\bigr)$.
Pach and Woeginger~\cite{PacWoe1990} subsequently showed that this bound is best possible up to a multiplicative constant implicit in the $O(\cdot)$ notation.
Rather tight two-sided bounds on the optimal constant, for every fixed $d \ge 2$ and $\varepsilon \to 0$, were later established by Komlós, Pach, and Woeginger~\cite{KomPacWoe1992}.

The constructions witnessing the lower bounds in~\cite{KomPacWoe1992,PacWoe1990} were purely combinatorial, and for a long time it was believed that `natural' geometrically defined range spaces should admit smaller $\varepsilon$-nets.
This belief was supported by upper bounds of order $1/\varepsilon$ proved for a number of basic two-dimensional range spaces, notably half-planes and disks~\cite{MatSeiWel1990}; see also the discussion in~\cite{PacTar2013}.
Other geometric families exhibit intermediate behaviour.
For axis-parallel rectangles in the plane and for axis-parallel boxes in $\mathbb{R}^3$, Aronov, Ezra, and Sharir~\cite{AroEzrSha2010} proved an upper bound of order $1/\varepsilon \cdot \log \log(1/\varepsilon)$.
It thus came as a surprise when Pach and Tardos~\cite{PacTar2013} found several geometric range spaces of VC-dimension two that require $\varepsilon$-nets of size $\Omega\bigl(1/\varepsilon \cdot \log(1/\varepsilon)\bigr)$; one such example is given by axis-parallel boxes in $\mathbb{R}^4$ having a vertex at the origin.

For the particularly elementary range space of straight lines in the plane, however, the optimal order of the minimum size of an $\varepsilon$-net remained unknown.
Given a finite set $P\subseteq\mathbb{R}^2$, define
\[
  \cL(P)\coloneqq\{P\cap \ell:\ell \text{ is a line in } \mathbb{R}^2\};
\]
clearly, $\mathcal{L}(P)$ has VC-dimension at most two.
Settling a problem raised by Matou{\v{s}}ek, Seidel, and Welzl~\cite{MatSeiWel1990}, Alon~\cite{Alo2012} constructed planar point sets $P$ for which every $\eps$-net for $\mathcal{L}(P)$ has size $\omega(1/\varepsilon)$.
Alon's elegant argument is based on the density Hales--Jewett theorem of Furstenberg and Katznelson~\cite{FurKat1991}, and the function implicit in the $\omega(\cdot)$ notation therefore tends to infinity extremely slowly.
Notwithstanding, Alon suggested that there should in fact be planar point sets requiring $\varepsilon$-nets of size $\Omega\bigl(1/\eps \cdot \log(1/\eps)\bigr)$.
A substantial breakthrough was achieved by Balogh and Solymosi~\cite{BalSol2018}, who used the hypergraph container method~\cite{BalMorSam15,BalMorSam18,SaxTho15} to prove that there are planar point sets requiring $\eps$-nets of size at least $1/\varepsilon \cdot (\log(1/\varepsilon))^{1/3-o(1)}$.
Subsequently, Balogh and the second author~\cite{BalSam2020} improved the $1/3$ in the exponent to $1/2$ by revisiting the argument of~\cite{BalSol2018}, armed with an efficient version of the hypergraph container lemma (which was the main result of~\cite{BalSam2020}).
We prove the optimal lower bound, confirming Alon's prediction.

\begin{theorem}
  \label{thm:points-in-the-plane}
  There is a constant $c>0$ such that, for every $\eps_0>0$, there are $\eps\in(0,\eps_0)$ and a finite set $P \subseteq \mathbb{R}^2$ such that every $\eps$-net for $\mathcal{L}(P)$ has at least $c/\eps \cdot \log(1/\varepsilon)$ points.
\end{theorem}

In fact, our main result is a lower bound on the smallest size of an $\varepsilon$-net that applies to a~large class of linear\footnote{Recall that a hypergraph is called linear if any two of its distinct edges intersect in at most one vertex, which also implies that its VC-dimension cannot exceed two.}, uniform hypergraphs.
For a hypergraph $\mathcal{H}$ on a finite vertex set $V$ and a~set $R \subseteq V$, write $\mathcal{H}(R) \coloneqq \{A \cap R : A \in \mathcal{H}\}$, cf.~the definition of $\mathcal{L}(P)$ presented above.
Given a subhypergraph $\mathcal{G} \subseteq \mathcal{H}$, which we shall always treat as spanning, we denote its average (vertex) degree by $d(\mathcal{G})$ and its maximum (vertex) degree by $\Delta(\mathcal{G})$.
Further, $\mathcal{G} \subseteq \mathcal{H}$ will be called \emph{$K$-almost-regular}, for some $K \ge 1$, if $\Delta(\mathcal{G}) \le K \cdot d(\mathcal{G})$.
Finally, given a $p \in (0,1]$, we denote by $V_p$ the random subset of $V$ obtained by independently retaining each element of $V$ with probability~$p$.
The following theorem is the main result of this work.

\begin{theorem}
  \label{thm:main}
  There is an absolute constant $C$ such that, for every $K \ge 1$, there exists an $\alpha = \alpha(K) > 0$ with the following property.
  Let $\mathcal{H}$ be a nonempty, $r$-uniform, $K$-almost-regular, linear hypergraph with vertex set $V$, let $\varepsilon \coloneqq r/(4|V|)$, and suppose that $p \in (0,2/3]$ satisfies
  \[
    (1-p)^r \cdot \Delta(\mathcal{H}) \ge C \log|V|
    \qquad
    \text{and}
    \qquad
    p |V| \ge CK.
  \]
  Then, with positive probability, every $\varepsilon$-net for $\mathcal{H}(V_p)$ has cardinality at least $\alpha p |V|$.
\end{theorem}

To deduce \cref{thm:points-in-the-plane}, it suffices to construct arbitrarily large sets $V \subseteq \mathbb{R}^2$ for which $\mathcal{L}(V)$ contains a nonempty, almost-regular subhypergraph of uniformity $\Theta(\log |V|)$ that satisfies the assumptions of \cref{thm:main} with $p = \Theta(1)$.
We show that, for infinitely many $n$, the grid $V \coloneqq \br{n}^2$ has this property.\footnote{Throughout this note, we use the shorthand notation $\br{n} \coloneqq \{1, \dotsc, n\}$.}
While this is quite straightforward to check, we postpone the details to the end of this note.

\subsection*{Organisation}

The remainder of this note is organised as follows:
In \cref{sec:concentration}, we recall two well-known concentration inequalities needed in the proof of \cref{thm:main}, which is then presented in \cref{sec:proof-main-theorem}.
The short \cref{sec:almost-regular-construction} details how \cref{thm:points-in-the-plane} can be derived from \cref{thm:main}.

\subsection*{Statement on the use of AI}

An earlier version of our argument used the hypergraph container lemma from~\cite{CamSam2026} and gave a lower bound of order $1/\varepsilon \cdot \log(1/\varepsilon) \cdot \bigl(\log\log(1/\varepsilon)\bigr)^{-1}$.
During the preparation of this manuscript, we asked OpenAI's ChatGPT for comments on a preliminary draft.
It suggested replacing the container argument with an averaging argument whose streamlined version is the current proof of \cref{lemma:good-net-unlikely}.

\subsection*{Acknowledgement}

The second author would like to thank Noga Alon for several inspiring conversations on this topic.

\section{Probabilistic estimates}
\label{sec:concentration}

The proof of~\cref{thm:main} requires two well-known concentration inequalities.
First, we shall use the following standard upper-tail estimate for binomial distributions.

\begin{lemma}
  \label{lemma:binomial-upper-tail}
  For every integer $N \ge 1$ and all $q \in [0,1]$,
  \[
    \Pr\bigl(\Bin(N,q)\ge 2qN\bigr) \le \exp(-qN/3).
  \]
\end{lemma}

We shall also need the following straightforward corollary of Bernstein's inequality for sums of independent bounded random variables (see, e.g., \cite[Theorem~2.9.5]{Ver2026}).

\begin{lemma}
  \label{lemma:bernstein-lower-tail}
  Let $M > 0$ and let $X_1,\dotsc,X_N$ be independent random variables taking values in $[0,M]$.
  Writing $X\coloneqq\sum_{i=1}^N X_i$ and $\mu\coloneqq\Ex[X]$, we have
  \[
    \Pr(X\le\mu/2)
    \le \exp\left(-\frac{\mu}{10M}\right).
  \]
\end{lemma}

\section{Proof of the main theorem}
\label{sec:proof-main-theorem}

Let $\mathcal{H}$ be as in the statement of \cref{thm:main} and let $\bR \sim V_p$.
Write $N \coloneqq |V|$ and, for every $R \subseteq V$, let
\[
  \mathcal{H}-R \coloneqq \{A \in \mathcal{H} : A \cap R = \emptyset\};
\]
we shall regard $\mathcal{H} - R$ as a spanning subhypergraph of $\mathcal{H}$.
Our first key lemma states that the random subhypergraph $\mathcal{H}-\bR$ is likely to remain almost regular.

\begin{lemma}
  \label{lemma:almost-regular}
  With probability at least $1/50$, the hypergraph $\mathcal{H}-\bR$ is nonempty and $12K$-almost-regular.
\end{lemma}

A set $S \subseteq V$ will be called \emph{$D$-good}, for some $D \ge 1$, if $\cH-S$ contains a nonempty $D$-almost-regular subhypergraph spanning $V$.
Note, crucially, that being $D$-good is a decreasing property, that is, any subset of a $D$-good set is $D$-good as well.
Our second key lemma states that, when $S \subseteq V$ is good, the random set $\bR$ is extremely unlikely to intersect every edges of $\mathcal{H} - S$ in fewer than $pr/2$ points.
Formally, for a subset $S \subseteq V$, define
\[
  \mathcal{M}_S \coloneqq \{R \subseteq V : |A \cap R| < pr/2 \text{ for every } A \in \mathcal{H} - S\}.
\]

\begin{lemma}
  \label{lemma:good-net-unlikely}
  For every $D \ge 1$ and every $D$-good $S \subseteq V$,
  \[
    \Pr(\bR \in \mathcal{M}_S) \le \exp\bigl(-pN / (10D)\bigr).
  \]
\end{lemma}

Let us first show how \cref{lemma:almost-regular,lemma:good-net-unlikely} imply the assertion of \cref{thm:main}.
Set $D \coloneqq 12K$, let $\alpha \in (0,1)$ be a constant satisfying $40 D \alpha \log(e/\alpha) \le 1$, let $\varepsilon \coloneqq r/(4N)$, and let $\mathcal{B}$ be the~family of all $R \subseteq V$ that contain an $\varepsilon$-net for $\mathcal{H}(R)$ with cardinality at most $\alpha p N$.
Observe that if $|R| \le 2pN$, then every $\varepsilon$-net $S \subseteq R$ for $\mathcal{H}(R)$ must intersect every $A \in \mathcal{H}$ that satisfies $|A \cap R| \ge pr/2 = 2 \varepsilon p N \ge \varepsilon |R|$, that is, $R \in \mathcal{M}_S$.
Therefore,
\begin{equation}
  \label{eq:key-upper-bound}
  \Pr(\bR \in \mathcal{B}) \le \Pr(|\bR| > 2pN) + \Pr(\bR \text{ is not $D$-good})
  + \sum_{\substack{S: \text{$D$-good} \\ |S| \le \alpha p N}} \Pr(S \subseteq \bR \text{ and } \bR \in \mathcal{M}_S).
\end{equation}

We now bound the three summands in the right-hand side of~\eqref{eq:key-upper-bound}.
By \cref{lemma:binomial-upper-tail}, the first summand is at most $\exp(- pN/3)$.
By \cref{lemma:almost-regular}, the second summand is at most $49/50$.
In order to bound the third summand (the sum), note first that, for every $S \subseteq V$, the events that $S \subseteq \bR$ and that $\bR \in \mathcal{M}_S$ are independent;
indeed, the latter event depends only on the intersection of $\bR$ with the edges of $\mathcal{H}-S$, which are all disjoint from $S$.
Since the sum ranges only over $D$-good sets, we may use \cref{lemma:good-net-unlikely} to bound it by
\[
  \Sigma \coloneqq \sum_{s \le \alpha p N} \binom{N}{s} \cdot p^s \cdot \exp\bigl(-pN / (10D)\bigr).
\]
Further, since
\[
  \sum_{s \le \alpha p N} \binom{N}{s} \cdot p^s \le \sum_{s \le \alpha p N} \left(\frac{e p N}{s}\right)^s \le (1+\alpha p N) \cdot (e/\alpha)^{\alpha p N} \le (e/\alpha)^{2\alpha p N} \le e^{pN/(20D)},
\]
thanks to our choice of $\alpha$, we may conclude that $\Sigma \le \exp(-pN/(20D))$.
Summarising,
\[
  \Pr(\bR \in \mathcal{B}) \le 49/50 + \exp(-pN/3) + \exp\bigl(-pN/(20D)\bigr).
\]
By our assumption that $pN \ge CK$ for a large constant $C$, we may conclude that $\Pr(\bR \in \mathcal{B}) < 1$, as desired.
It now remains to prove \Cref{lemma:almost-regular,lemma:good-net-unlikely}.

\begin{proof}[Proof of~\cref{lemma:almost-regular}]
  Fix an arbitrary $v \in V$ and condition on the event $v \notin \bR$.
  Since $\mathcal{H}$ is linear, the events $A \in \mathcal{H} - \bR$ are independent for all edges $A \in \mathcal{H}$ that contain $v$ and thus $\deg_{\cH-\bR} v \sim\Bin(\deg_{\cH}v,q)$, where $q \coloneqq (1-p)^{r-1}$.
  In particular, it follows from \cref{lemma:binomial-upper-tail} that
  \[
    \Pr\bigl(\deg_{\cH-\bR}v \ge 2q \cdot \Delta(\cH) \mid v \notin \bR\bigr) \le \exp\bigl(-q \cdot \Delta(\cH)/3\bigr).
  \]
  As $\deg_{\mathcal{H} - \bR} v = 0$ in case $v \in \bR$, the above upper-tail estimate holds also unconditionally.
  Thanks to our assumption that $q \cdot \Delta(\mathcal{H}) \ge C \log N$ for a large constant $C$, we may use the union bound over all $v \in V$ to conclude that $\Delta(\mathcal{H} - \bR) \le 2q \cdot \Delta(\mathcal{H})$ with probability at least $49/50$.
  
  Further, let $X \coloneqq |\mathcal{H}-\bR|$ and note that $\mu \coloneqq \Ex[X] = (1-p)^r \cdot |\mathcal{H}|$.\footnote{We shall identify hypergraphs with their sets of edges, so that $|\mathcal{H}|$ is the number of edges of $\mathcal{H}$.}
  Since the union of any two distinct edges of $\mathcal{H}$ has at least $2r-1$ vertices, by linearity, we have
  \[
    \Ex[X^2] = \sum_{A, B \in \mathcal{H}} \Pr\bigl((A \cup B) \cap \bR = \emptyset\bigr) = \sum_{A, B \in \mathcal{H}} (1-p)^{|A \cup B|} \le \mu + \frac{\mu^2}{1-p} \le \frac{2\mu^2}{1-p},
  \]
  where the last inequality holds as $\mu = (1-p)q \cdot |\mathcal{H}| \ge (1-p)q \cdot \Delta(\mathcal{H}) \ge 1-p$.
  Consequently, the Paley--Zygmund inequality gives
  \[
    \Pr(X\ge\mu/2) \ge \frac{\mu^2}{4\Ex[X^2]}\ge\frac{1-p}{8} \ge \frac{1}{24}.
  \]

  Finally, on the intersection of the two events
  \[
    \Delta(\mathcal{H} - \bR) \le 2q \cdot \Delta(\mathcal{H})
    \qquad
    \text{and}
    \qquad
    |\mathcal{H}-\bR| \ge \frac{\mu}{2} = \frac{(1-p)q \cdot |\mathcal{H}|}{2} \ge \frac{q|\mathcal{H}|}{6}
  \]
  which has probability at least $1/24 - 1/50$, we have
  \[
    \begin{split}
      \Delta(\mathcal{H}-\bR) \le 2q \cdot \Delta(\mathcal{H}) \le 2qK \cdot d(\mathcal{H}) = \frac{2qK \cdot r|\mathcal{H}|}{N} \le \frac{12K \cdot r|\mathcal{H}-\bR|}{N} = 12K \cdot d(\mathcal{H} - \bR),
    \end{split}
  \]
  as claimed.
\end{proof}

\begin{proof}[Proof of~\cref{lemma:good-net-unlikely}]
  Suppose that $S \subseteq V$ is $D$-good and let $\mathcal{G}$ be a nonempty $D$-almost-regular subhypergraph of $\mathcal{H} - S$.
  Define
  \[
    X \coloneqq \sum_{A\in\cG} |A \cap \bR| = \sum_{v \in V}\deg_{\cG}v \cdot \1_{v\in\bR}
  \]
  and note that $X$ is a sum of independent random variables, each taking values in
  $\{0, \dotsc, \Delta(\cG)\}$.
  Suppose now that $\bR \in \mathcal{M}_S$, that is, $|A \cap \bR| < pr/2$ for every $A \in \mathcal{H} - S \supseteq \mathcal{G}$.
  On this event, we clearly have $X < |\mathcal{G}| \cdot pr/2$ while
  \[
    \mu \coloneqq \Ex[X] = p \cdot r |\mathcal{G}| = p \cdot d(\mathcal{G}) \cdot N \ge \frac{p \cdot \Delta(\mathcal{G}) \cdot N}{D}
  \]
  and thus, by \cref{lemma:bernstein-lower-tail},
  \[
    \Pr(X \le \mu/2) \le \exp\left(-\frac{\mu}{10 \cdot \Delta(\cG)}\right) \le \exp\left(-\frac{pN}{10D}\right),
  \]
  as claimed.
\end{proof}

\section{Lower bound for epsilon nets for lines}
\label{sec:almost-regular-construction}

Fix some $\gamma \in (0,1/3)$ and let $M$ be a large prime.
Set
\[
  r\coloneqq 2 \lfloor\gamma \log M\rfloor,
  \quad
  n\coloneqq rM,
  \quad
  V\coloneqq\br{n}^2,
  \quad
  N \coloneqq n^2,
  \quad
  \text{and}
  \quad
  H\coloneqq\lfloor M/2 \rfloor
\]
and note that $n$ is even and $(\gamma / 2) \cdot \log N \le r \le \gamma \log N$, provided that $M$ is sufficiently large.
Given $(x,y) \in V$ and $h \in \br{H}$, denote by
\[
  \ell(x,y;h)
  \coloneqq
  \bigl\{(x,y)+t \cdot (h,M):t\in\mathbb{R}\bigr\}
\]
the line with slope $M/h$ that passes through $(x,y)$.
Finally, for $h \in \br{H}$, define
\[
  \mathcal{H}_h
  \coloneqq
  \bigl\{\ell(x,y;h) \cap V:
    x \in\br{n/2} \land y\in\br{M}\bigr\}
\]
and let $\mathcal{H} \coloneqq \mathcal{H}_1 \cup \dotsb \cup \mathcal{H}_H$, viewed as a hypergraph on $V$.

\begin{lemma}
  \label{lemma:hyp-lines-almost-regular}
  The hypergraph $\mathcal{H}$ is $r$-uniform, $2$-almost-regular, and satisfies $d(\mathcal{H}) \ge N^{1/3}$.
\end{lemma}
\begin{proof}
  Fix an arbitrary triple $x \in \br{n/2}$, $y \in \br{M}$, and $h \in \br{H}$.
  Since $M$ is prime, and thus $\gcd(h,M)=1$, the integer points of $\ell(x,y;h)$ are precisely the points $(x,y) + t \cdot (h,M)$ with $t \in \mathbb{Z}$.
  Further, since $1 \le y + t \cdot M \le rM = n$ if and only if $t \in \{0, \dotsc, r-1\}$ and
  \[
    1 \le x + t \cdot h \le n/2 + r \cdot H \le n
  \]
  for each such $t$, the line $\ell(x,y;h)$ intersects $V$ in precisely $r$ points;
  in other words, $\mathcal{H}$ is $r$-uniform.
  Further, since the second coordinate of only one of those points belongs to $\br{M}$, we have $|\mathcal{H}_h| = n/2 \cdot M$.
  Since the hypergraphs $\mathcal{H}_1, \dotsc, \mathcal{H}_H$ are pairwise disjoint, we have
  \[
    d(\mathcal{H}) = \frac{r \cdot |\mathcal{H}|}{N} = \frac{r \cdot H \cdot n/2 \cdot M}{n^2} = \frac{H}{2},
  \]
  while $\Delta(\mathcal{H}) \le H$, as every point of $V$ lies on at most one line of every given slope;
  this implies that $\mathcal{H}$ is $2$-almost-regular.
  Finally, since $r \le \gamma \log N$, we have $d(\mathcal{H}) \ge M/5 \ge n^{2/3} = N^{1/3}$ for sufficiently large $M$.
\end{proof}

We are now ready to deduce \cref{thm:points-in-the-plane}.
We first verify that $\mathcal{H}$ satisfies the assumptions of \cref{thm:main} with $K \coloneqq 2$ and $p \coloneqq 1/2$.
Indeed, $p |V| \ge M^2 / 2 \ge 2C$ and
\[
  (1-p)^r \cdot \Delta(\mathcal{H}) \ge (1-p)^r \cdot d(\mathcal{H}) \ge 2^{-\gamma \log N} \cdot N^{1/3} \ge N^{1/3-\gamma} \ge C \log N
\]
when $M$, and hence also $N$, is sufficiently large.
Now, let $\alpha = \alpha(2)$ be the constant supplied by~\cref{thm:main}, let $\varepsilon \coloneqq r/(4N)$, and note that $\varepsilon \to 0$ as $M \to \infty$.
The theorem yields some $P \subseteq V \subseteq \mathbb{R}^2$ such that every $\varepsilon$-net for $\mathcal{H}(P)$, and thus also every $\varepsilon$-net for $\mathcal{L}(P) \supseteq \mathcal{H}(P)$, has at least $\alpha N /2$ points.
Since
\[
  \frac{1}{\varepsilon} \log \frac{1}{\varepsilon} = \frac{4N}{r} \log \frac{4N}{r} \le \frac{4N \log N}{r} \le \frac{8N}{\gamma}
\]
for sufficiently large $M$, we have $\alpha N / 2 \ge (\alpha \gamma / 16) \cdot 1/\varepsilon \cdot \log(1/\varepsilon)$, as desired.

\bibliographystyle{abbrv}
\bibliography{eps-nets}

@article{Alo2012,
  author  = {Alon, Noga},
  title   = {A Non-linear Lower Bound for Planar {$\varepsilon$}-Nets},
  journal = {Discrete \& Computational Geometry},
  volume  = {47},
  number  = {2},
  pages   = {235--244},
  year    = {2012},
  doi     = {10.1007/s00454-010-9323-7}
}

@article{AroEzrSha2010,
  author  = {Aronov, Boris and Ezra, Esther and Sharir, Micha},
  title   = {Small-Size {$\varepsilon$}-Nets for Axis-Parallel Rectangles and Boxes},
  journal = {SIAM Journal on Computing},
  volume  = {39},
  number  = {7},
  pages   = {3248--3282},
  year    = {2010},
  doi     = {10.1137/090762968}
}

@article{BalMorSam15,
  author  = {Balogh, J{\'o}zsef and Morris, Robert and Samotij, Wojciech},
  title   = {Independent Sets in Hypergraphs},
  journal = {Journal of the American Mathematical Society},
  volume  = {28},
  number  = {3},
  pages   = {669--709},
  year    = {2015},
  doi     = {10.1090/S0894-0347-2014-00816-X}
}

@article{BalSam2020,
  author  = {Balogh, J{\'o}zsef and Samotij, Wojciech},
  title   = {An Efficient Container Lemma},
  journal = {Discrete Analysis},
  volume  = {2020},
  pages   = {Paper No. 17, 56 pp.},
  year    = {2020},
  doi     = {10.19086/da.17354}
}

@article{BalSol2018,
  author  = {Balogh, J{\'o}zsef and Solymosi, J{\'o}zsef},
  title   = {On the Number of Points in General Position in the Plane},
  journal = {Discrete Analysis},
  volume  = {2018},
  pages   = {Paper No. 16, 20 pp.},
  year    = {2018},
  doi     = {10.19086/da.4438}
}

@article{CheVap1971,
  author  = {Vapnik, Vladimir N. and Chervonenkis, Alexey Ya.},
  title   = {On the Uniform Convergence of Relative Frequencies of Events to Their Probabilities},
  journal = {Theory of Probability and Its Applications},
  volume  = {16},
  number  = {2},
  pages   = {264--280},
  year    = {1971},
  doi     = {10.1137/1116025}
}

@article{FurKat1991,
  author  = {Furstenberg, Hillel and Katznelson, Yitzhak},
  title   = {A Density Version of the {Hales--Jewett} Theorem},
  journal = {Journal d'Analyse Math{\'e}matique},
  volume  = {57},
  number  = {1},
  pages   = {64--119},
  year    = {1991},
  doi     = {10.1007/BF03041066}
}

@article{HauWel1987,
  author  = {Haussler, David and Welzl, Emo},
  title   = {{$\varepsilon$}-Nets and Simplex Range Queries},
  journal = {Discrete \& Computational Geometry},
  volume  = {2},
  number  = {2},
  pages   = {127--151},
  year    = {1987},
  doi     = {10.1007/BF02187876}
}

@article{KomPacWoe1992,
  author  = {Koml{\'o}s, J{\'a}nos and Pach, J{\'a}nos and Woeginger, Gerhard J.},
  title   = {Almost Tight Bounds for {$\varepsilon$}-Nets},
  journal = {Discrete \& Computational Geometry},
  volume  = {7},
  number  = {1},
  pages   = {163--173},
  year    = {1992},
  doi     = {10.1007/BF02187833}
}

@inproceedings{MatSeiWel1990,
  author    = {Matou{\v{s}}ek, Ji{\v{r}}{\'i} and Seidel, Raimund and Welzl, Emo},
  title     = {How to Net a Lot with Little: Small {$\varepsilon$}-Nets for Disks and Halfspaces},
  booktitle = {Proceedings of the Sixth Annual Symposium on Computational Geometry},
  pages     = {16--22},
  publisher = {Association for Computing Machinery},
  address   = {New York},
  year      = {1990},
  doi       = {10.1145/98524.98530}
}

@inproceedings{PacWoe1990,
  author    = {Pach, J{\'a}nos and Woeginger, Gerhard J.},
  title     = {Some New Bounds for {$\varepsilon$}-Nets},
  booktitle = {Proceedings of the Sixth Annual Symposium on Computational Geometry},
  pages     = {10--15},
  publisher = {Association for Computing Machinery},
  address   = {New York},
  year      = {1990},
  doi       = {10.1145/98524.98529}
}

@article{PacTar2013,
  author  = {Pach, J{\'a}nos and Tardos, G{\'a}bor},
  title   = {Tight Lower Bounds for the Size of {$\varepsilon$}-Nets},
  journal = {Journal of the American Mathematical Society},
  volume  = {26},
  number  = {3},
  pages   = {645--658},
  year    = {2013},
  doi     = {10.1090/S0894-0347-2012-00759-0}
}

@article{SaxTho15,
  author  = {Saxton, David and Thomason, Andrew},
  title   = {Hypergraph Containers},
  journal = {Inventiones Mathematicae},
  volume  = {201},
  number  = {3},
  pages   = {925--992},
  year    = {2015},
  doi     = {10.1007/s00222-014-0562-8}
}

@book{Ver2026,
  author    = {Vershynin, Roman},
  title     = {High-Dimensional Probability: An Introduction with Applications in Data Science},
  edition   = {2nd},
  publisher = {Cambridge University Press},
  series    = {Cambridge Series in Statistical and Probabilistic Mathematics},
  volume    = {58},
  year      = {2026},
  isbn      = {9781009490641},
  doi       = {10.1017/9781009490672}
}

@article{CamSam2026,
  author  = {Campos, Marcelo and Samotij, Wojciech},
  title   = {Towards an Optimal Hypergraph Container Lemma},
  journal = {Combinatorica},
  volume  = {46},
  pages   = {Article 24, 31 pp.},
  year    = {2026},
  doi     = {10.1007/s00493-026-00214-1}
}

@inproceedings {BalMorSam18,
    AUTHOR = {Balogh, J\'ozsef and Morris, Robert and Samotij, Wojciech},
     TITLE = {The method of hypergraph containers},
 BOOKTITLE = {Proceedings of the {I}nternational {C}ongress of
              {M}athematicians---{R}io de {J}aneiro 2018. {V}ol. {IV}.
              {I}nvited lectures},
     PAGES = {3059--3092},
 PUBLISHER = {World Sci. Publ., Hackensack, NJ},
      YEAR = {2018},
      ISBN = {978-981-3272-93-4; 978-981-3272-87-3},
   MRCLASS = {05-02 (05C35 05C65 05D10)},
  MRNUMBER = {3966523},
}

\end{document}